\documentclass{article}

\usepackage{amssymb,amsmath,color}
\usepackage{amsthm}
\usepackage{url}

\definecolor{red}{rgb}{1,0,0}
\definecolor{blue}{rgb}{.2,.2,.8}

\newtheorem{theorem}{Theorem}[section]
\newtheorem{lemma}{Lemma}
\newtheorem{corollary}[theorem]{Corollary}

\theoremstyle{remark}
\newtheorem*{remark}{Remark}

\begin{document}

\title{Dyadic Structure of the Two-Block Odd Partition Function and Theta-Weighted Congruences for
	\(\operatorname{pod}(n)\)}
\author{Mircea Merca\\
	\footnotesize Department of Mathematical Methods and Models,\\ 
	\footnotesize Fundamental Sciences Applied in Engineering Research Center,\\ 
	\footnotesize National University of Science and Technology Politehnica Bucharest,\\
	\footnotesize RO-060042, Bucharest, Romania\\
	\footnotesize Academy of Romanian Scientists, RO-050044, Bucharest, Romania\\
	\footnotesize mircea.merca@upb.ro
}
\date{}
\maketitle

\begin{abstract}
	The two-block odd partition function \(a(n)\) is the signed
enumeration of partitions into exactly two distinct part sizes, each
occurring an odd number of times. By separating the underlying
representations according to the \(2\)-adic valuations of the two part
sizes, we obtain a signed decomposition, completed by two classical
theta-function evaluations, which explains the structural relation
\(
a(2^km)
=
a(m)+(2^{k-1}-1)\sigma(m),
\)
where \(m\) is odd, \(k\ge1\), and
\(\sigma(m)\) is the sum-of-divisors function.
Combined with a previously established arithmetic formula for \(a(n)\),
this relation yields congruences along the dyadic progressions
\(2^k(4n+3)\), as well as the families
\(
a\!\left(2^\alpha3^\beta(12n+11)\right)\equiv0\pmod3
\)
and
\(
a\!\left(2^\alpha(18n+15)\right)\equiv0\pmod3.
\)
We also use a factorization of the generating function of \(a(n)\)
involving \(\operatorname{pod}(n)\), the number of partitions in which
odd parts are distinct and even parts are unrestricted, to establish a
uniform reduction modulo every odd prime of the associated auxiliary
series. The resulting theta-weighted sums involving
\(\operatorname{pod}(n)\) are congruent to \(a(n)\) and therefore admit
explicit evaluations involving the sum-of-divisors function and the
character divisor sum associated with the nonprincipal Dirichlet
character modulo \(4\). As special cases, we obtain infinite families
of weighted congruences modulo \(3\), \(5\), and \(13\).
\end{abstract}

\section{Introduction}\label{S1}

Arithmetic properties of partition functions are frequently reflected
in congruences, divisor-sum formulas, and theta-series identities.
In particular, explicit connections between divisor functions and
partition congruences have proved useful in revealing arithmetic
structures that are not apparent from the defining generating
functions alone; see, for example, \cite{Merca24,Merca26}.

Recently, the author introduced in \cite{Merca} the two-block odd
partition function \(a(n)\), defined by
\[
A(q):=\sum_{n\ge1}a(n)q^n
=
\sum_{1\le n_1<n_2}
\frac{q^{n_1+n_2}}
{(1+q^{2n_1})(1+q^{2n_2})}.
\]
Its coefficients count, with signs, partitions into exactly two
distinct part sizes, each occurring an odd number of times.

One of the main results of \cite{Merca} is the explicit arithmetic
formula
\begin{align}
a(n)
=
\frac{\sigma(n)-\lambda(n)}4
-\frac{\sigma(n/2)}2
+\sigma(n/4),                                      \label{eq1}
\end{align}
where
\[
\sigma(n):=\sum_{d\mid n}d,
\qquad
\lambda(n):=\sum_{d\mid n}\chi_4(d),
\]
and
\[
\chi_4(n):=
\begin{cases}
1, & n\equiv1\pmod4,\\
-1, & n\equiv3\pmod4,\\
0, & n\equiv0,2\pmod4.
\end{cases}
\]
As usual, we set
\(
\sigma(x)=0
\)
whenever \(x\) is not a positive integer.

Formula~\eqref{eq1} gives a complete arithmetic description of
\(a(n)\), but it does not explain how the underlying signed
representations change when the argument is multiplied by a power of
\(2\). Our first objective is to provide such an explanation.

By separating the signed representations according to whether the two
distinct part sizes have equal or different \(2\)-adic valuations, we
obtain the following structural relation.

\begin{theorem}\label{thm:structural-relation}
	Let \(m\) be an odd positive integer and let \(k\ge1\). Then
	\[
	a(2^km)
	=
	a(m)+(2^{k-1}-1)\sigma(m).
	\]
\end{theorem}

Although this identity can be derived algebraically from
formula~\eqref{eq1}, the new content of
Theorem~\ref{thm:structural-relation} lies in the decomposition of the
underlying signed representations. Those for which the two part sizes
have different \(2\)-adic valuations are in a weight-preserving
bijection with the representations counted by \(a(m)\), whereas the
representations with equal \(2\)-adic valuations have total signed
weight
$(2^{k-1}-1)\sigma(m)$.
The evaluation of this latter contribution is completed by two
classical theta-function identities.

Combining Theorem~\ref{thm:structural-relation} with
formula~\eqref{eq1} gives
\[
a(2^km)
=
\frac{
	(2^{k+1}-3)\sigma(m)-\lambda(m)
}{4}.
\]
If
$m\equiv3\pmod4$,
then \(\lambda(m)=0\), and hence
$4a(2^km)=(2^{k+1}-3)\sigma(m)$.
This yields the following dyadic family.

\begin{corollary}\label{Th1.3}
	Let \(m\) be a positive integer satisfying
	\(
	m\equiv3\pmod4,
	\)
	and let \(k\ge2\). Then
	\[
	a(2^km)\equiv0\pmod{2^{k+1}-3}.
	\]
\end{corollary}

Taking \(m=4n+3\) and \(k=2,3,4,5\) yields the first four
instances, with moduli \(5\), \(13\), \(29\), and \(61\), respectively;
the corresponding congruences are recorded in
Corollary~\ref{Cor1.4}.

The arithmetic formula also gives two families of congruences modulo
\(3\). They follow from elementary divisibility properties of
\(\sigma(n)\) and from the vanishing of \(\lambda(n)\) on suitable
arithmetic progressions.

\begin{corollary}\label{thm:general-mod3-first}
	For all integers \(\alpha,\beta,n\ge0\),
	\[
	a\!\left(
	2^\alpha3^\beta(12n+11)
	\right)
	\equiv0\pmod3.
	\]
\end{corollary}

\begin{corollary}\label{thm:general-mod3-second}
	For all integers \(\alpha,n\ge0\),
	\[
	a\!\left(
	2^\alpha(18n+15)
	\right)
	\equiv0\pmod3.
	\]
\end{corollary}

The case \(\alpha=\beta=0\) of the first congruence was established in
\cite{Merca} by a different method. Formula~\eqref{eq1} shows that it
belongs to the larger family above.

Our second objective is to transfer the arithmetic information carried
by \(a(n)\) to weighted sums involving the partition function
\(\operatorname{pod}(n)\). Recall that
\(\operatorname{pod}(n)\) counts partitions in which odd parts are
distinct and even parts are unrestricted. Its arithmetic properties
have been studied extensively; see
\cite{BMpod,CGM15,FXY20,HS10,RS11,W15}.

Let
\[
P(q):=
\sum_{n\ge0}\operatorname{pod}(n)q^n.
\]
In \cite{Merca}, the generating function \(A(q)\) was shown to admit
the factorization
\[
A(q)=P(q)S(q),
\]
where
\[
S(q)
=
\sum_{r=1}^{\infty}
(-1)^{T_{r-1}}
T_{\lceil r/2\rceil}
q^{T_{r+1}},
\qquad
T_j:=\frac{j(j+1)}2.
\]

For an odd prime \(p\), define
\[
\varepsilon_s
:=
(-1)^{T_{s-1}}
T_{\lceil s/2\rceil}
\]
and, for \(N\ge1\),
\[
\mathcal W_p(N)
:=
\sum_{s=1}^{2p-4}
\varepsilon_s
\sum_{j\in\mathbb Z}
\operatorname{pod}
\!\left(
N-T_{4pj+s+1}
\right).
\]
We adopt the convention
\[
\operatorname{pod}(N)=0
\qquad(N<0),
\]
so that all bilateral sums are finite.

The coefficients of \(S(q)\) are periodic modulo \(p\), with period
dividing \(4p\). Pairing the corresponding residue classes gives a
uniform reduction of \(S(q)\) modulo every odd prime. Combined with
formula~\eqref{eq1}, this produces the following explicit transfer
principle.

\begin{theorem}\label{thm:explicit-transfer}
	Let \(p\) be an odd prime. Then, for every integer \(N\ge1\),
	\[
	\mathcal W_p(N)
	\equiv
	\frac{\sigma(N)-\lambda(N)}4
	-\frac{\sigma(N/2)}2
	+\sigma(N/4)
	\pmod p.
	\]
	In particular, if \(m\) is an odd positive integer and
	\(\kappa\ge1\), then
	\[
	\mathcal W_p(2^\kappa m)
	\equiv
	\frac{
		(2^{\kappa+1}-3)\sigma(m)-\lambda(m)
	}{4}
	\pmod p.
	\]
	Here division by \(2\) and \(4\) is understood modulo \(p\).
\end{theorem}

Thus the weighted \(\operatorname{pod}\)-sum is not merely related to
\(a(N)\); it admits an explicit evaluation in terms of the arithmetic
functions \(\sigma\) and \(\lambda\). In particular, the dyadic
specialization gives the following family of congruences.

\begin{corollary}\label{cor:general-pod-dyadic}
	Let \(p\ge5\) be a prime, and let \(\kappa\ge2\) satisfy
	$2^{\kappa+1}\equiv3\pmod p$.
	Then, for every integer \(n\ge0\),
	\[
	\sum_{s=1}^{2p-4}
	\varepsilon_s
	\sum_{j\in\mathbb Z}
	\operatorname{pod}
	\!\left(
	2^\kappa(4n+3)-T_{4pj+s+1}
	\right)
	\equiv0\pmod p.
	\]
\end{corollary}

For \(p=5\), the condition on \(\kappa\) is equivalent to
$\kappa\equiv2\pmod4$,
whereas for \(p=13\) it is equivalent to
$\kappa\equiv3\pmod{12}$.
These specializations produce infinite families of weighted
congruences modulo \(5\) and \(13\).

The condition in Corollary~\ref{cor:general-pod-dyadic} has no solution
for \(p=3\). In that case, a more compact reduction of \(S(q)\),
combined with Corollaries~\ref{thm:general-mod3-first} and
\ref{thm:general-mod3-second}, gives two further families of weighted
congruences involving \(\operatorname{pod}(n)\).

The paper is organized as follows. Section~\ref{S2} records the
properties of \(\lambda(n)\) needed in the sequel.
Section~\ref{S3} develops the signed decomposition underlying
Theorem~\ref{thm:structural-relation}. Section~\ref{S4} derives its
dyadic congruence consequences, while Section~\ref{S5} treats the two
families modulo \(3\). Finally, Section~\ref{S6} establishes the
reduction of \(S(q)\) modulo odd primes, proves
Theorem~\ref{thm:explicit-transfer}, and derives the weighted
congruences for \(\operatorname{pod}(n)\).

\section{Arithmetic preliminaries}
\label{S2}

We record three elementary properties of the arithmetic function
\(\lambda(n)\) that are relevant to the dyadic structure studied
in this paper.

\smallskip
\begin{lemma}\label{lem:lambda-properties}
	The following properties hold:
	\begin{enumerate}
		\item[(i)] The function \(\lambda(n)\) is multiplicative.
		
		\item[(ii)] If \(m\) is odd, then
		$
		\lambda(2^k m)=\lambda(m)
		$
		for every integer \(k\ge0\).
		
		\item[(iii)] If \(m\) is odd and
		$m\equiv3\pmod4$,
		then
		$\lambda(m)=0$.
	\end{enumerate}
\end{lemma}

\begin{proof}
	Since
	$\lambda=\mathbf{1}*\chi_4$,
	where \(\mathbf{1}(n)=1\) and \(*\) denotes Dirichlet convolution,
	part~(i) follows from the multiplicativity of
	\(\mathbf{1}\) and \(\chi_4\).
	
	For part~(ii), every divisor of \(2^km\) has the form
	$2^j d$, $0\le j\le k$, $d\mid m$.
	Since
	\[
	\chi_4(2^jd)=0
	\qquad(j\ge1),
	\]
	only the odd divisors contribute to \(\lambda(2^km)\). Therefore
	\[
	\lambda(2^km)
	=
	\sum_{d\mid m}\chi_4(d)
	=
	\lambda(m).
	\]
	
	For part~(iii), write
	\[
	m=\prod_{p^\alpha\parallel m}p^\alpha.
	\]
	Because \(m\equiv3\pmod4\), at least one prime
	\(p\equiv3\pmod4\) occurs to an odd exponent \(\alpha\). By
	multiplicativity,
	\[
	\lambda(m)
	=
	\prod_{p^\alpha\parallel m}
	\left(
	1+\chi_4(p)+\cdots+\chi_4(p)^\alpha
	\right).
	\]
	For the prime \(p\equiv3\pmod4\) occurring to an odd exponent,
	\[
	1+\chi_4(p)+\cdots+\chi_4(p)^\alpha
	=
	1-1+\cdots-1
	=
	0.
	\]
	Hence \(\lambda(m)=0\).
\end{proof}

\section{A signed partition interpretation}
\label{S3}

For \(n\ge1\), let \(\mathcal A_n\) denote the set of all quadruples
\[
(x,y,r,s)\in \mathbb N^2\times \mathbb N_{0}^2
\]
satisfying
\[
x<y,\qquad n=(2r+1)x+(2s+1)y.
\]
We attach to such a quadruple the weight
\[
w(x,y,r,s):=(-1)^{r+s}.
\]
Expanding the defining generating function of \(a(n)\), we obtain
\[
\begin{aligned}
A(q)
&=\sum_{1\le x<y}
\frac{q^{x+y}}{(1+q^{2x})(1+q^{2y})} \\
&=\sum_{1\le x<y}\sum_{r,s\ge0}
(-1)^{r+s}q^{(2r+1)x+(2s+1)y}.
\end{aligned}
\]
Hence
\[
a(n)=\sum_{(x,y,r,s)\in\mathcal A_n} w(x,y,r,s).
\]


Define
\[
F(q):=\sum_{\substack{n\ge1\\ n\ {\rm odd}}}
\frac{q^n}{1+q^{2n}}.
\]

\begin{lemma}\label{lem:jacobi-odd-square}
	We have
	\[
	F(q)=\sum_{\substack{n\ge1\\ n\ {\rm odd}}}\lambda(n)\,q^n
	\]
	and, for every odd positive integer \(m\),
	\[
	[q^{2^L m}]F(q)^2=
	\begin{cases}
	\sigma(m), & L=1,\\
	0, & L\ge2.
	\end{cases}
	\]
\end{lemma}

\begin{proof}
	Expanding each geometric series gives
	\[
	F(q)=
	\sum_{\substack{n\ge1\\ n\ {\rm odd}}}
	\sum_{r\ge0}(-1)^r q^{(2r+1)n}.
	\]
	Thus, for \(n\) odd, the coefficient of \(q^n\) in \(F(q)\) is
	\[
	\sum_{d\mid n}\chi_4(d)=\lambda(n),
	\]
	which proves the first identity.
	
	For the second assertion, let
	\[
	\psi(q)=\sum_{n\ge0}q^{n(n+1)/2}
	\]
	be Ramanujan's theta function. 	According to \cite[Theorem 3.44, eq. (3.111)]{Cooper}, we have
	\[
	\psi(q)^2 = \sum_{n\ge 0} \lambda(4n+1)\,q^n.
	\]
	So, by Lemma~\ref{lem:lambda-properties}(iii),
	\[
	F(q)=
	q\sum_{n\ge0}\lambda(4n+1)\,q^{4n} =q\,\psi(q^4)^2.
	\]
	According to \cite[Theorem 3.44, eq. (3.112)]{Cooper}, we have
	\[
	\psi(q)^4 = \sum_{n\ge 0} \sigma(2n+1)\,q^n.
	\]
	Hence
	\[
	F(q)^2 = q^2\,\psi(q^4)^4 = \sum_{n\ge0} \sigma(2n+1)\,q^{4n+2}.
	\]
	
	The coefficient of \(q^{2^L m}\) in $F(q)^2$ is nonzero only when
	$2^L m=4n+2$,
	or equivalently
	$2^{L-1}m=2n+1$.
	Since \(m\) is odd, this forces \(L=1\), and then \(n=(m-1)/2\). In this
	case the coefficient is \(\sigma(m)\). For \(L\ge2\), no such exponent
	occurs, and the coefficient is \(0\).
\end{proof}

For \(n\ge1\), let \(\mathcal A'_n\) denote the set of all quadruples
\[
(x,y,r,s)\in (2\mathbb N-1)^2\times \mathbb N_0^2
\]
satisfying
\[
x<y,
\qquad
n=(2r+1)x+(2s+1)y.
\]
For \(L\ge1\) and for an odd positive integer \(m\), define
\[
C_L(m):=
\sum_{(x,y,r,s)\in \mathcal A'_{2^L m}}
(-1)^{r+s}.
\]

\begin{lemma}\label{lem:odd-level-contribution}
	Let \(m\) be an odd positive integer. Then
	\[
	C_L(m)=
	\begin{cases}
	0, & L=1,\\
	2^{L-2}\sigma(m), & L\ge2.
	\end{cases}
	\]
\end{lemma}

\begin{proof}
	The series
	\[
	F(q)^2
	=
	\sum_{x,y\ {\rm odd}}
	\frac{q^{x+y}}
	{(1+q^{2x})(1+q^{2y})}
	\]
	is the signed generating function for all ordered pairs
	\((x,y)\) of odd part sizes, including the diagonal
	\(x=y\).
	
	Indeed, the coefficient of \(q^n\) in \(F(q)^2\) is
	\[
	\sum_{\substack{(2r+1)x+(2s+1)y=n\\x,y\ {\rm odd}\\ r,s\ge0}}
	(-1)^{r+s},
	\]
	where the ordered pairs \((x,y)\) and \((y,x)\) are counted
	separately whenever \(x\neq y\).
	
	Our quantity \(C_L(m)\), however, counts only representations
	with distinct odd part sizes satisfying \(x<y\).
	Hence we first remove the diagonal contribution
	\[
	D(q):=
	\sum_{x\ {\rm odd}}
	\frac{q^{2x}}{(1+q^{2x})^2},
	\]
	which corresponds exactly to the case \(x=y\).
	
	The remaining representations occur in symmetric pairs
	\((x,y)\) and \((y,x)\), each having the same weight
	\((-1)^{r+s}\). Therefore, for every odd positive integer
	\(m\) and every \(L\ge1\),
	\[
	C_L(m)
	=
	\frac12
	[q^{2^Lm}]
	\bigl(F(q)^2-D(q)\bigr).
	\]
	
	By Lemma~\ref{lem:jacobi-odd-square},
	\[
	[q^{2m}]F(q)^2=\sigma(m),
	\qquad
	[q^{2^L m}]F(q)^2=0
	\quad (L\ge2).
	\]
	
	It remains to compute the coefficients of \(D(q)\). Since
	\[
	\frac{z}{(1+z)^2}
	=
	\sum_{j\ge1}(-1)^{j-1}jz^j,
	\]
	we obtain
	\[
	D(q)
	=
	\sum_{x\ {\rm odd}}
	\sum_{j\ge1}
	(-1)^{j-1}j\,q^{2jx}.
	\]
	
	Hence the coefficient of \(q^{2n}\) in \(D(q)\) is
	\[
	[q^{2n}]D(q)
	=
	\sum_{\substack{x\mid n\\ x\ {\rm odd}}}
	(-1)^{n/x-1}\frac{n}{x}.
	\]
	
	Now write
	\(
	n=2^{L-1}m,
	\)
	where \(m\) is odd.
	If \(L=1\), then \(n=m\) is odd, so \(n/x\) is odd for every divisor
	\(x\mid m\). Consequently,
	\[
	[q^{2m}]D(q)
	=
	\sum_{x\mid m}\frac{m}{x}
	=
	\sigma(m).
	\]
	Assume now that \(L\ge2\). Since \(x\) is odd, every odd divisor of
	\(n\) is a divisor of \(m\), and
	\[
	\frac{n}{x}
	=
	2^{L-1}\frac{m}{x}
	\]
	is even. Therefore
	\[
	[q^{2^Lm}]D(q)
	=
	-\sum_{x\mid m}
	2^{L-1}\frac{m}{x}
	=
	-2^{L-1}\sigma(m).
	\]
	
	Combining the two coefficient evaluations gives
	\[
	C_1(m)
	=
	\frac12\bigl(\sigma(m)-\sigma(m)\bigr)
	=
	0,
	\]
	and, for \(L\ge2\),
	\[
	C_L(m)
	=
	\frac12\bigl(0-(-2^{L-1}\sigma(m))\bigr)
	=
	2^{L-2}\sigma(m).
	\]
	This completes the proof.
\end{proof}

\begin{proof}[\textbf{A signed proof of
		Theorem~\ref{thm:structural-relation}}]
	Fix an odd positive integer \(m\) and let \(k\ge1\). For a positive
	integer \(u\), let \(v_2(u)\) denote the exponent of \(2\) in its
	prime factorization. Define
	\[
	\mathcal A_{2^km}^{\ne}
	=
	\left\{
	(x,y,r,s)\in\mathcal A_{2^km}:
	v_2(x)\ne v_2(y)
	\right\}
	\]
	and
	\[
	\mathcal A_{2^km}^{=}
	=
	\left\{
	(x,y,r,s)\in\mathcal A_{2^km}:
	v_2(x)=v_2(y)
	\right\}.
	\]
	These two sets form a disjoint partition of
	\(\mathcal A_{2^km}\).
	
	Every positive integer admits a unique decomposition
	$x=2^\alpha X$, $y=2^\beta Y$, 
	where \(X\) and \(Y\) are odd.
	
	\medskip
	\noindent
	\textit{Case 1. \(v_2(x)\ne v_2(y)\).}
	Suppose that \(\alpha\ne\beta\), and put
	\(
	\gamma=\min(\alpha,\beta).
	\)
	Then
	\[
	\begin{aligned}
	(2r+1)x+(2s+1)y
	={}&
	2^\gamma
	\bigl(
	(2r+1)2^{\alpha-\gamma}X
	+(2s+1)2^{\beta-\gamma}Y
	\bigr).
	\end{aligned}
	\]
	Exactly one of the integers
	$2^{\alpha-\gamma}$, $2^{\beta-\gamma}$	is equal to \(1\), while the other is even. Since \(X,Y,2r+1\),
	and \(2s+1\) are odd, the expression inside the parentheses is
	odd. Therefore
	\[
	v_2\bigl((2r+1)x+(2s+1)y\bigr)=\gamma.
	\]
	Because
	\[
	(2r+1)x+(2s+1)y=2^km
	\]
	and \(m\) is odd, it follows that
	\(
	\gamma=k.
	\)
	
	Consequently, both \(x\) and \(y\) are divisible by \(2^k\), and
	the map
	\[
	\Phi:
	\mathcal A_{2^km}^{\ne}\longrightarrow\mathcal A_m,
	\qquad
	\Phi(x,y,r,s)
	=
	\left(
	\frac{x}{2^k},
	\frac{y}{2^k},
	r,s
	\right),
	\]
	is well defined. Indeed,
	\[
	m
	=
	(2r+1)\frac{x}{2^k}
	+
	(2s+1)\frac{y}{2^k},
	\]
	and division by \(2^k\) preserves the inequality \(x<y\).
	
	We now verify that \(\Phi\) is a bijection. Let
	$(u,v,r,s)\in\mathcal A_m$
	and write
	$u=2^cU$, $v=2^dV$,
	where \(U\) and \(V\) are odd. We first observe that \(c\ne d\).
	Indeed, if \(c=d\), then
	\[
	m
	=
	2^c\bigl((2r+1)U+(2s+1)V\bigr),
	\]
	and the expression inside the parentheses is even. This would imply
	that \(m\) is even, contrary to our assumption.
	
	Since \(c\ne d\), the same \(2\)-adic argument used above gives
	$v_2(m)=\min(c,d)$.
	As \(m\) is odd, we obtain
	$\min(c,d)=0$.
	It follows that
	\[
	(2^ku,2^kv,r,s)\in\mathcal A_{2^km}^{\ne}.
	\]
	Thus the inverse of \(\Phi\) is
	\[
	(u,v,r,s)
	\longmapsto
	(2^ku,2^kv,r,s).
	\]
	The bijection preserves the weight \((-1)^{r+s}\). Therefore
	\[
	\sum_{(x,y,r,s)\in\mathcal A_{2^km}^{\ne}}
	(-1)^{r+s}
	=
	\sum_{(u,v,r,s)\in\mathcal A_m}
	(-1)^{r+s}
	=
	a(m).
	\]
	
	\medskip
	\noindent
	\textit{Case 2. \(v_2(x)=v_2(y)\).}
	Suppose now that
	\(
	\alpha=\beta=t.
	\)
	Then
	$x=2^tX$, $y=2^tY$,
	where \(X<Y\) are odd. The defining equation becomes
	\[
	2^km
	=
	2^t\bigl((2r+1)X+(2s+1)Y\bigr).
	\]
	The expression inside the parentheses is a sum of two odd integers
	and is therefore even. Consequently, \(t<k\), so
	\(
	0\le t\le k-1.
	\)
	
	For a fixed value of \(t\), put
	\(
	L=k-t.
	\)
	Dividing the defining equation by \(2^t\) gives
	\[
	2^Lm=(2r+1)X+(2s+1)Y.
	\]
	Hence the map
	\[
	(x,y,r,s)
	\longmapsto
	(X,Y,r,s)
	\]
	is a weight-preserving bijection between the representations in
	\(\mathcal A_{2^km}^{=}\) with common \(2\)-adic valuation \(t\)
	and the set \(\mathcal A'_{2^Lm}\). Its inverse is
	\[
	(X,Y,r,s)
	\longmapsto
	(2^tX,2^tY,r,s).
	\]
	Therefore, the total signed contribution corresponding to the fixed
	value \(t\) is \(C_L(m)\).
	
	As \(t\) runs from \(0\) to \(k-1\), the integer \(L=k-t\) runs
	from \(k\) to \(1\). Consequently,
	\[
	\sum_{(x,y,r,s)\in\mathcal A_{2^km}^{=}}
	(-1)^{r+s}
	=
	\sum_{L=1}^{k}C_L(m).
	\]
	By Lemma~\ref{lem:odd-level-contribution},
	\[
	C_1(m)=0,
	\qquad
	C_L(m)=2^{L-2}\sigma(m)
	\quad(L\ge2).
	\]
	It follows that
	\[
	\begin{aligned}
	\sum_{(x,y,r,s)\in\mathcal A_{2^km}^{=}}
	(-1)^{r+s}
	&=
	\sum_{L=2}^{k}2^{L-2}\sigma(m)
	=
	(2^{k-1}-1)\sigma(m),
	\end{aligned}
	\]
	where the sum is understood to be empty when \(k=1\).
	
	\medskip
	Combining the contributions from the two cases, we obtain
	\[
	\begin{aligned}
	a(2^km)
	&=
	\sum_{(x,y,r,s)\in\mathcal A_{2^km}^{\ne}}
	(-1)^{r+s}
	+
	\sum_{(x,y,r,s)\in\mathcal A_{2^km}^{=}}
	(-1)^{r+s}\\
	&=
	a(m)+(2^{k-1}-1)\sigma(m).
	\end{aligned}
	\]
	This proves Theorem~\ref{thm:structural-relation}.
\end{proof}

\smallskip
\begin{remark}
	The preceding proof gives a signed partition-theoretic interpretation of the correction term
	\(
	(2^{k-1}-1)\sigma(m).
	\)
	Indeed, this term is exactly the total signed contribution of those representations of \(2^k m\) for which the two part sizes have the same \(2\)-adic valuation. In contrast, the representations for which the two part sizes have different \(2\)-adic valuations are in weight-preserving bijection with the signed representations counted by \(a(m)\).
\end{remark}
\smallskip

\smallskip
\begin{corollary}\label{Th1.1}
	Let \(m\) be an odd positive integer and let \(k\ge1\). Then
	\[
	a(2^k m)
	=
	\frac{(2^{k+1}-3)\sigma(m)-\lambda(m)}4.
	\]
\end{corollary}

\begin{proof}
	Since \(m\) is odd, formula~\eqref{eq1} gives
	\[
	a(m)=\frac{\sigma(m)-\lambda(m)}4.
	\]
	By Theorem~\ref{thm:structural-relation},
	\[
	a(2^k m)
	=
	a(m)+(2^{k-1}-1)\sigma(m).
	\]
	Consequently,
	\[
	\begin{aligned}
	a(2^k m)
	&=
	\frac{\sigma(m)-\lambda(m)}4
	+(2^{k-1}-1)\sigma(m)
	=
	\frac{(2^{k+1}-3)\sigma(m)-\lambda(m)}4.
	\end{aligned}
	\]
\end{proof}

\begin{remark}
	Corollary~\ref{Th1.1} may also be verified directly from the
	arithmetic formula~\eqref{eq1}. Indeed,
	Lemma~\ref{lem:lambda-properties}(ii) gives
	\(
	\lambda(2^km)=\lambda(m).
	\)
	Moreover, since \(m\) is odd, the multiplicativity of \(\sigma\)
	gives, for every \(k\ge1\),
	\[
	\sigma(2^km)
	=
	(2^{k+1}-1)\sigma(m),
	\]
	\[
	\sigma(2^km/2)
	=
	(2^k-1)\sigma(m),
	\]
	and
	\[
	\sigma(2^km/4)
	=
	(2^{k-1}-1)\sigma(m),
	\]
	where the last identity also holds for \(k=1\), by the convention
	\(\sigma(m/2)=0\). Substitution into~\eqref{eq1} yields
	\[
	a(2^km)
	=
	\frac{(2^{k+1}-3)\sigma(m)-\lambda(m)}4.
	\]
	Theorem~\ref{thm:structural-relation}, however, gives the additional
	signed partition-theoretic interpretation of this formula.
\end{remark}

\smallskip
\begin{corollary}\label{Cor1.2}
	Let \(m\) be an odd positive integer satisfying
	\(
	m\equiv3\pmod4.
	\)
	Then, for every integer \(k\ge1\),
	\[
	a(2^k m)
	=
	\frac{2^{k+1}-3}{4}\sigma(m).
	\]
\end{corollary}

\begin{proof}
	By Corollary~\ref{Th1.1} and
	Lemma~\ref{lem:lambda-properties}(iii), we have
	\[
	a(2^k m)
	=
	\frac{(2^{k+1}-3)\sigma(m)-\lambda(m)}4
	=
	\frac{2^{k+1}-3}{4}\sigma(m).
	\]
\end{proof}

\section{Congruences on dyadic progressions}
\label{S4}

We now derive an infinite family of congruences on arithmetic progressions
obtained by fixing the odd part of the argument and varying its \(2\)-adic
valuation. The key point is that, for integers \(m\equiv3\pmod4\),
Corollary~\ref{Cor1.2} expresses \(a(2^k m)\) as a constant multiple of
\(\sigma(m)\), with the constant depending only on \(k\).

\smallskip
\begin{proof}[\textbf{Proof of Corollary~\ref{Th1.3}}]
	By Corollary~\ref{Cor1.2}, we have
	$4a(2^k m)=(2^{k+1}-3)\sigma(m)$.
	Hence
	$4a(2^k m)\equiv0\pmod{2^{k+1}-3}$.
	Since \(2^{k+1}-3\) is odd, \(4\) is invertible modulo
	\(2^{k+1}-3\). Therefore 
	$a(2^k m)\equiv0\pmod{2^{k+1}-3}$.
\end{proof}

\smallskip
\begin{corollary}\label{Cor1.4}
	For every integer \(n\ge0\),
	\begin{enumerate}
		\item
		\(
		a(16n+12)\equiv0\pmod5;
		\)
		\item
		\(
		a(32n+24)\equiv0\pmod{13};
		\)
		\item
		\(
		a(64n+48)\equiv0\pmod{29};
		\)
		\item
		\(
		a(128n+96)\equiv0\pmod{61}.
		\)
	\end{enumerate}
\end{corollary}

\begin{proof}
	We apply Corollary~\ref{Th1.3} with
	\(
	m=4n+3.
	\)
	For \(k=2,3,4,5\), we have respectively
	\[
	2^{k+1}-3=5,\ 13,\ 29,\ 61.
	\]
	Thus Corollary~\ref{Th1.3} gives
	\[
	a(2^2(4n+3))\equiv0\pmod5,
	\]
	\[
	a(2^3(4n+3))\equiv0\pmod{13},
	\]
	\[
	a(2^4(4n+3))\equiv0\pmod{29},
	\]
	and
	\[
	a(2^5(4n+3))\equiv0\pmod{61}.
	\]
	Therefore the desired congruences follow.
\end{proof}

Corollary~\ref{Th1.3} shows that the congruence properties of the
two-block odd partition function along the progressions
\(
2^k(4n+3)
\)
are governed by the divisors of the integers
\(
2^{k+1}-3.
\)
Consequently, Corollary~\ref{Th1.3} immediately yields congruences modulo every positive divisor of 
$2^{k+1}-3$. In particular, whenever a positive divisor $\ell$ of $2^{k+1}-3$ is fixed, one obtains
$a\big(2^k(4n+3)\big)\equiv 0 \pmod \ell$,
for every $n\ge 0$.

\section{Congruences modulo {3}}
\label{S5}

In this section we prove the congruences modulo \(3\) announced in
Corollaries~\ref{thm:general-mod3-first} and
\ref{thm:general-mod3-second}. We begin with an elementary congruence for the
divisor-sum function.

\smallskip
\begin{lemma}\label{lem:sigma3}
	Let \(n\) be a positive integer satisfying
	\(
	n\equiv2\pmod3.
	\)	Then
	\[
	\sigma(n)\equiv0\pmod3.
	\]
\end{lemma}

\begin{proof}
	Write
	\[
	n=\prod_{i=1}^{r}p_i^{\alpha_i}.
	\]
	Since
	\(
	n\equiv2\pmod3,
	\)
	the total exponent of the prime factors congruent to
	\(2\pmod3\) is odd. Hence there exists a prime
	\(p\equiv2\pmod3\) such that
	$p^\alpha\parallel n$,
	where \(\alpha\) is odd.
	
	Since the divisor-sum function is multiplicative,
	\[
	\sigma(n)
	=
	\prod_{i=1}^{r}\sigma(p_i^{\alpha_i}),
	\]
	it is sufficient to prove that
	\(
	\sigma(p^\alpha)\equiv0\pmod3.
	\)
	
	Because
	\(
	p\equiv-1\pmod3,
	\)
	we have
	\[
	\begin{aligned}
	\sigma(p^\alpha)
	&=
	1+p+\cdots+p^\alpha\\
	&\equiv
	1-1+\cdots-1
	\pmod3.
	\end{aligned}
	\]
	Since \(\alpha\) is odd, the above alternating sum contains an even number
	of terms and is therefore congruent to \(0\) modulo \(3\). Hence
	\(
	\sigma(n)\equiv0\pmod3.
	\)
\end{proof}

\begin{proof}[\textbf{Proof of
		Corollary~\ref{thm:general-mod3-first}}]
	Put
	$u=12n+11$,
	$m=3^\beta u$.
	Since
	\(
	u\equiv2\pmod3,
	\)
	we have
	\(
	(3^\beta,u)=1.
	\)
	Moreover,
	\(
	u\equiv3\pmod4,
	\)
	so Lemma~\ref{lem:lambda-properties}(iii) gives
	\(
	\lambda(u)=0.
	\)
	By Lemma~\ref{lem:lambda-properties}(i),
	\[
	\lambda(m)
	=
	\lambda(3^\beta)\lambda(u)
	=
	0.
	\]
	Since \(u\equiv2\pmod3\), Lemma~\ref{lem:sigma3} gives
	\(
	\sigma(u)\equiv0\pmod3.
	\)
	The multiplicativity of \(\sigma\) therefore yields
	\[
	\sigma(m)
	=
	\sigma(3^\beta)\sigma(u)
	\equiv0\pmod3.
	\]
	If \(\alpha=0\), formula~\eqref{eq1} gives
	\[
	a(m)
	=
	\frac{\sigma(m)-\lambda(m)}4
	\equiv0\pmod3.
	\]
	Assume now that \(\alpha\ge1\). By
	Corollary~\ref{Th1.1},
	\[
	a(2^\alpha m)
	=
	\frac{(2^{\alpha+1}-3)\sigma(m)-\lambda(m)}4.
	\]
	Since
	$\sigma(m)\equiv0\pmod3$
	and
	$\lambda(m)=0$,
	we conclude that
	\(
	a(2^\alpha m)\equiv0\pmod3.
	\)
	This completes the proof.
\end{proof}

\begin{proof}[\textbf{Proof of
		Corollary~\ref{thm:general-mod3-second}}]
	Put
	$u=6n+5$, $m=3u=18n+15$.
	Then \(u\) and \(m\) are odd. Moreover,
	$u\equiv2\pmod3$,
	so
	$(3,u)=1$.
	By Lemma~\ref{lem:sigma3},
	$\sigma(u)\equiv0\pmod3$.
	The multiplicativity of \(\sigma\) therefore gives
	\[
	\sigma(m)
	=
	\sigma(3)\sigma(u)
	=
	4\sigma(u)
	\equiv0\pmod3.
	\]
	On the other hand,
	\[
	\lambda(3)
	=
	\chi_4(1)+\chi_4(3)
	=
	1-1
	=
	0.
	\]
	Since \((3,u)=1\), Lemma~\ref{lem:lambda-properties}(i) gives
	\[
	\lambda(m)
	=
	\lambda(3)\lambda(u)
	=
	0.
	\]
	If \(\alpha=0\), formula~\eqref{eq1} gives
	\[
	a(m)
	=
	\frac{\sigma(m)-\lambda(m)}4
	\equiv0\pmod3.
	\]
	Assume now that \(\alpha\ge1\). By
	Corollary~\ref{Th1.1},
	\[
	a(2^\alpha m)
	=
	\frac{(2^{\alpha+1}-3)\sigma(m)-\lambda(m)}4.
	\]
	Since
	$\sigma(m)\equiv0\pmod3$
	and
	$\lambda(m)=0$,
	and \(4\) is invertible modulo \(3\), we obtain
	$a(2^\alpha m)\equiv0\pmod3$.
	This completes the proof.
\end{proof}

Both congruences are direct consequences of the explicit arithmetic
formula for \(a(n)\), combined with elementary properties of
\(\sigma(n)\) and \(\lambda(n)\). In particular, the
congruence
\[
a(12n+11)\equiv0\pmod3,
\]
previously established in \cite{Merca} by considerably more involved
arguments, is seen to be an immediate consequence of the explicit formula
\eqref{eq1}.

\section{Weighted congruences involving
	\(\operatorname{pod}(n)\)}
\label{S6}

We now apply the arithmetic properties of \(a(n)\) to weighted sums
involving \(\operatorname{pod}(n)\). Recall from \cite{Merca} the
factorization
\[
A(q)=P(q)S(q),
\]
where
\[
P(q)
=
\sum_{n\ge0}\operatorname{pod}(n)q^n
\]
and
\[
S(q)
=
\sum_{r=1}^{\infty}
(-1)^{T_{r-1}}
T_{\lceil r/2\rceil}
q^{T_{r+1}},
\qquad
T_j:=\frac{j(j+1)}2
\quad(j\in\mathbb Z).
\]

We shall use the symmetry
$T_j=T_{-j-1}$.
As in the Introduction, we adopt the convention
\[
\operatorname{pod}(n)=0
\qquad(n<0).
\]
Consequently, all bilateral sums occurring below are finite for each
fixed coefficient.

\subsection{A general reduction modulo odd primes}

We first establish the periodicity modulo \(p\) of the coefficients
appearing in the auxiliary series \(S(q)\).

\begin{lemma}\label{lem:general-periodicity}
	Let \(p\) be an odd prime and define
	\[
	\varepsilon_r
	=
	(-1)^{T_{r-1}}
	T_{\lceil r/2\rceil},
	\qquad r\ge1.
	\]
	Then
	\[
	\varepsilon_{r+4p}
	\equiv
	\varepsilon_r
	\pmod p
	\]
	for every \(r\ge1\).
\end{lemma}

\begin{proof}
	We first compare the signs. Since
	$T_{r+4p-1}-T_{r-1}	= 	2p(2r+2p-1)$,
	the difference is even. Hence
	$(-1)^{T_{r+4p-1}}=(-1)^{T_{r-1}}$.
	Next,
	$\lceil(r+4p)/2\rceil = \lceil r/2 \rceil+2p$. 
	Put
	$m=\lceil r/2 \rceil$.
	Then
	\[
	T_{m+2p}-T_m
	=
	2pm+p(2p+1),
	\]
	which is divisible by \(p\). Therefore
	\[
	T_{m+2p}\equiv T_m\pmod p.
	\]
	
	Combining these two congruences gives
	\[
	\begin{aligned}
	\varepsilon_{r+4p}
	&=
	(-1)^{T_{r+4p-1}}
	T_{\lceil(r+4p)/2\rceil}\\
	&\equiv
	(-1)^{T_{r-1}}
	T_{\lceil r/2\rceil}\\
	&=
	\varepsilon_r
	\pmod p.
	\end{aligned}
	\]
\end{proof}

The preceding periodicity allows us to identify the vanishing residue
classes and to pair the remaining classes.

\begin{lemma}\label{lem:general-reduction}
	Let \(p\) be an odd prime. Then, as a coefficientwise congruence of
	formal power series,
	\[
	S(q)
	\equiv
	\sum_{s=1}^{2p-4}
	\varepsilon_s
	\sum_{k\in\mathbb Z}
	q^{T_{4pk+s+1}}
	\pmod p,
	\]
	where
	$\varepsilon_s = (-1)^{T_{s-1}} T_{\lceil s/2\rceil}$.
\end{lemma}

\begin{proof}
	By Lemma~\ref{lem:general-periodicity}, the coefficients
	\(\varepsilon_r\) are periodic modulo \(p\), with period dividing
	\(4p\). We first determine the residue classes for which
	\(\varepsilon_r\) vanishes modulo \(p\).
	
	For \(1\le r\le4p\), put
	$m=\lceil r/2 \rceil$.
	Then
	$\varepsilon_r = (-1)^{T_{r-1}}T_m$,
	and therefore
	$$\varepsilon_r\equiv0\pmod p \iff
	T_m\equiv0\pmod p.$$
	Since \(p\) is odd and prime, this is equivalent to
	$m\equiv0,-1\pmod p$.
	
	Because \(1\le m\le2p\), the possible values of \(m\) are
	\[
	m=p-1,\quad p,\quad 2p-1,\quad 2p.
	\]
	Consequently, the vanishing residue classes of \(r\) modulo \(4p\)
	are
	\[
	2p-3,\quad 2p-2,\quad 2p-1,\quad 2p,
	\]
	and
	\[
	4p-3,\quad 4p-2,\quad 4p-1,\quad 0.
	\]
	
	It remains to pair the nonvanishing residue classes. For
	$1\le s\le2p-4$,
	put
	$s^{*}=4p-s-3$.
	We claim that
	$\varepsilon_{s^{*}} \equiv \varepsilon_s \pmod p$.
	
	Let
	$m=\lceil s/2 \rceil$.
	A direct calculation, according as \(s\) is odd or even, gives
	\[
	\lceil s^{*}/2 \rceil
	=
	2p-m-1.
	\]
	It follows that
	\[
	T_{\lceil s^{*}/2\rceil}
	=
	T_{2p-m-1}
	\equiv
	T_m
	\pmod p.
	\]
	
	Moreover,
	\[
	\begin{aligned}
	T_{s^{*}-1}-T_{s-1}
	&=
	T_{4p-s-4}-T_{s-1}
	=
	2(p-1)(4p-2s-3),
	\end{aligned}
	\]
	which is even. Hence
	$ (-1)^{T_{s^{*}-1}} = (-1)^{T_{s-1}}$,
	and therefore
	$ \varepsilon_{s^{*}} \equiv \varepsilon_s \pmod p$.
	
	Using the periodicity of \(\varepsilon_r\), the vanishing residue
	classes found above, and the pairing
	$s\longleftrightarrow4p-s-3$,
	we obtain
	\[
	\begin{aligned}
	S(q)
	\equiv
	\sum_{s=1}^{2p-4}
	\varepsilon_s
	\left(
	\sum_{k\ge0}q^{T_{4pk+s+1}}
	+
	\sum_{k\ge0}q^{T_{4pk+4p-s-2}}
	\right)
	\pmod p.
	\end{aligned}
	\]
	
	Finally, the identity
	\(
	T_j=T_{-j-1}
	\)
	gives
	\(
	T_{4pk+4p-s-2}
	=
	T_{-4p(k+1)+s+1}.
	\)
	After replacing \(k\) by \(-k-1\) in the second unilateral sum, the
	two sums combine into
	\[
	\sum_{k\in\mathbb Z}q^{T_{4pk+s+1}}.
	\]
	Consequently,
	\[
	S(q)
	\equiv
	\sum_{s=1}^{2p-4}
	\varepsilon_s
	\sum_{k\in\mathbb Z}
	q^{T_{4pk+s+1}}
	\pmod p.
	\]
\end{proof}

\subsection{An explicit transfer principle}

We now prove the explicit transfer principle stated in the
Introduction. Recall the notation \(\mathcal W_p(N)\) introduced
there for the corresponding theta-weighted
\(\operatorname{pod}\)-sum.

\begin{proof}[\textbf{Proof of
		Theorem~\ref{thm:explicit-transfer}}]
	By Lemma~\ref{lem:general-reduction} and the factorization
	\(A(q)=P(q)S(q)\), we have
	\[
	A(q)
	\equiv
	P(q)
	\sum_{s=1}^{2p-4}
	\varepsilon_s
	\sum_{j\in\mathbb Z}
	q^{T_{4pj+s+1}}
	\pmod p.
	\]
	Comparing the coefficients of \(q^N\) gives
	\[
	\mathcal W_p(N)\equiv a(N)\pmod p.
	\]
	Formula~\eqref{eq1} therefore yields
	\[
	\mathcal W_p(N)
	\equiv
	\frac{\sigma(N)-\lambda(N)}4
	-\frac{\sigma(N/2)}2
	+\sigma(N/4)
	\pmod p,
	\]
	which proves the first assertion.
	
	Now let
	\(
	N=2^\kappa m,
	\)
	where \(m\) is odd and \(\kappa\ge1\). By
	Corollary~\ref{Th1.1},
	\[
	a(2^\kappa m)
	=
	\frac{
		(2^{\kappa+1}-3)\sigma(m)-\lambda(m)
	}{4}.
	\]
	Combining this identity with
	$$\mathcal W_p(2^\kappa m) \equiv a(2^\kappa m) \pmod p$$
	proves the dyadic specialization.
\end{proof}

The dyadic congruence stated in the Introduction is now immediate.

\begin{proof}[\textbf{Proof of
		Corollary~\ref{cor:general-pod-dyadic}}]
	Put
	$m=4n+3$.
	Then
	$m\equiv3\pmod4$,
	so Lemma~\ref{lem:lambda-properties}(iii) gives
	$\lambda(m)=0$.
	The dyadic specialization of
	Theorem~\ref{thm:explicit-transfer} therefore yields
	\[
	\mathcal W_p\!\left(2^\kappa(4n+3)\right)
	\equiv
	\frac{
		(2^{\kappa+1}-3)\sigma(4n+3)
	}{4}
	\pmod p.
	\]
	By assumption,
	$2^{\kappa+1}-3\equiv0\pmod p$.
	Since \(4\) is invertible modulo \(p\), it follows that
	\[
	\mathcal W_p\!\left(2^\kappa(4n+3)\right)
	\equiv0\pmod p.
	\]
	By the definition of \(\mathcal W_p(N)\), this is precisely the
	asserted congruence.
\end{proof}

The congruence
$2^{\kappa+1}\equiv3\pmod3$
has no solutions. We therefore treat \(p=3\) separately, using a more
compact form of the reduction in
Lemma~\ref{lem:general-reduction}.

\subsection{The case \(p=3\)}

Applying Lemma~\ref{lem:general-reduction} with \(p=3\), we obtain
\[
S(q)
\equiv
\sum_{k\in\mathbb Z}q^{T_{12k+2}}
-
\sum_{k\in\mathbb Z}q^{T_{12k+3}}
\pmod3.
\]
This reduction can be written in the more compact form used in
\cite{Merca}. Indeed, separating the even and odd values of the
summation index gives
\[
\begin{aligned}
\sum_{k\in\mathbb Z}
(-1)^k q^{18k^2+15k+3}
&=
\sum_{j\in\mathbb Z}q^{T_{12j+2}}
-
\sum_{j\in\mathbb Z}q^{T_{12j+8}}
=
\sum_{j\in\mathbb Z}q^{T_{12j+2}}
-
\sum_{j\in\mathbb Z}q^{T_{12j+3}},
\end{aligned}
\]
where the last equality follows from
\[
T_{12j+8}
=
T_{-12j-9}
=
T_{12(-j-1)+3},
\]
followed by the change of index \(j\mapsto-j-1\).
Consequently,
\[
S(q)
\equiv
\sum_{k\in\mathbb Z}
(-1)^k q^{18k^2+15k+3}
\pmod3.
\]

Combining this congruence with the factorization
\(
A(q)=P(q)S(q),
\)
we obtain
\[
A(q)
\equiv
\left(
\sum_{n\ge0}
\operatorname{pod}(n)q^n
\right)
\left(
\sum_{k\in\mathbb Z}
(-1)^k
q^{18k^2+15k+3}
\right)
\pmod3.
\]
Equating coefficients of \(q^N\) yields
\[
a(N)
\equiv
\sum_{k\in\mathbb Z}
(-1)^k
\operatorname{pod}
\!\left(
N-18k^2-15k-3
\right)
\pmod3.
\]

Applying Corollaries~\ref{thm:general-mod3-first} and \ref{thm:general-mod3-second}, we obtain the following congruences for
\(\operatorname{pod}(n)\).

\smallskip
\begin{theorem}\label{thm:pod3}
	For all integers \(\alpha,\beta,n\ge0\), we have
	\[
	\sum_{k\in\mathbb Z}
	(-1)^k
	\operatorname{pod}
	\!\left(
	2^\alpha3^\beta(12n+11)
	-18k^2-15k-3
	\right)
	\equiv0
	\pmod3.
	\]
	Moreover, for all integers \(\alpha,n\ge0\),
	\[
	\sum_{k\in\mathbb Z}
	(-1)^k
	\operatorname{pod}
	\!\left(
	2^\alpha(18n+15)
	-18k^2-15k-3
	\right)
	\equiv0
	\pmod3.
	\]
\end{theorem}

\begin{proof}
	Taking
	$N=2^\alpha3^\beta(12n+11)$
	in the preceding coefficient identity and applying
	Corollary~\ref{thm:general-mod3-first} proves the first congruence.
	Similarly, taking
	$N=2^\alpha(18n+15)$
	and applying Corollary~\ref{thm:general-mod3-second} proves the second
	congruence.
\end{proof}

The congruences previously obtained by taking
\(N=12n+11\) and \(N=18n+15\) are the initial cases of
the two families in Theorem~\ref{thm:pod3}.

\subsection{The case \(p=5\)}

Applying Lemma~\ref{lem:general-reduction} with \(p=5\), we have
\(2p-4=6\). Reducing the coefficients
\[
\varepsilon_s
=
(-1)^{T_{s-1}}T_{\lceil s/2\rceil},
\qquad 1\le s\le6,
\]
modulo \(5\), gives
\[
(\varepsilon_1,\ldots,\varepsilon_6)
\equiv
(1,-1,2,-2,1,-1)
\pmod5.
\]
Hence, putting
\[
(c_1,\ldots,c_6)=(1,-1,2,-2,1,-1),
\]
we obtain
\[
S(q)
\equiv
\sum_{s=1}^{6}
c_s
\sum_{k\in\mathbb Z}
q^{T_{20k+s+1}}
\pmod5.
\]

The congruence
$2^{\kappa+1}\equiv3\pmod5$
holds precisely when
$\kappa\equiv2\pmod4$.
Therefore Corollary~\ref{cor:general-pod-dyadic} gives the following
infinite family.

\smallskip
\begin{theorem}\label{thm:pod5}
	For all integers \(j,n\ge0\), we have
	\[
	\sum_{s=1}^{6}
	\sum_{\ell\in\mathbb Z}
	c_s
	\operatorname{pod}
	\!\left(
	2^{4j+2}(4n+3)-T_{20\ell+s+1}
	\right)
	\equiv0
	\pmod5,
	\]
	where
	\[
	(c_1,c_2,c_3,c_4,c_5,c_6)
	=
	(1,-1,2,-2,1,-1).
	\]
\end{theorem}

\begin{proof}
	Since
	\[
	2^{4j+3}
	=
	2^3(2^4)^j
	\equiv
	3
	\pmod5,
	\]
	the integer
	$\kappa=4j+2$
	satisfies
	$2^{\kappa+1}\equiv3\pmod5$.
	The result therefore follows from
	Corollary~\ref{cor:general-pod-dyadic}, applied with \(p=5\).
\end{proof}

\begin{remark}
	The case \(j=0\) recovers the weighted congruence associated with the
	progression \(16n+12\).
\end{remark}

\subsection{The case \(p=13\)}

Applying Lemma~\ref{lem:general-reduction} with \(p=13\), we have
\(2p-4=22\). Reducing
\[
\varepsilon_s
=
(-1)^{T_{s-1}}T_{\lceil s/2\rceil},
\qquad 1\le s\le22,
\]
modulo \(13\), gives
\[
\begin{aligned}
(\varepsilon_1,\ldots,\varepsilon_{22})
\equiv{}&
(1,-1,-3,3,6,-6,3,-3,2,-2,5,-5,\\
&\hspace{10mm}
2,-2,3,-3,6,-6,-3,3,1,-1)
\pmod{13}.
\end{aligned}
\]
Thus, writing
\[
\begin{aligned}
(d_1,\ldots,d_{22})
={}&
(1,-1,-3,3,6,-6,3,-3,2,-2,5,-5,\\
&\hspace{20mm}
2,-2,3,-3,6,-6,-3,3,1,-1).
\end{aligned}
\]
we obtain
\[
S(q)
\equiv
\sum_{s=1}^{22}
d_s
\sum_{k\in\mathbb Z}
q^{T_{52k+s+1}}
\pmod{13}.
\]

Since
$\operatorname{ord}_{13}(2)=12$
and
$2^4\equiv3\pmod{13}$,
the congruence
$2^{\kappa+1}\equiv3\pmod{13}$
holds precisely when
$\kappa\equiv3\pmod{12}$.
Thus Corollary~\ref{cor:general-pod-dyadic} gives the following
infinite family.

\begin{theorem}\label{thm:pod13}
	For all integers \(j,n\ge0\), we have
	\[
	\sum_{s=1}^{22}
	\sum_{\ell\in\mathbb Z}
	d_s
	\operatorname{pod}
	\!\left(
	2^{12j+3}(4n+3)-T_{52\ell+s+1}
	\right)
	\equiv0
	\pmod{13},
	\]
	where
	\[
	\begin{aligned}
	(d_1,\ldots,d_{22})
	={}&
	(1,-1,-3,3,6,-6,3,-3,2,-2,5,-5,\\
	&\hspace{20mm}
	2,-2,3,-3,6,-6,-3,3,1,-1).
	\end{aligned}
	\]
\end{theorem}

\begin{proof}
	Since
	$2^{12}\equiv1\pmod{13}$
	and
	$2^4\equiv3\pmod{13}$,
	we have
	$2^{12j+4}\equiv3\pmod{13}$.
	Thus
	$\kappa=12j+3$
	satisfies
	$2^{\kappa+1}\equiv3\pmod{13}$.
	The result follows from
	Corollary~\ref{cor:general-pod-dyadic}, applied with \(p=13\).
\end{proof}

\begin{remark}
	The case \(j=0\) recovers the weighted congruence associated with the
	progression \(32n+24\).
\end{remark}

\section{Concluding remarks}
\label{S7}

The first main result of this paper is the signed decomposition
underlying the relation
\[
a(2^km)
=
a(m)+(2^{k-1}-1)\sigma(m),
\]
where \(m\) is odd and \(k\ge1\). The representations in which the two
part sizes have different \(2\)-adic valuations are in a
weight-preserving bijection with those counted by \(a(m)\), whereas
the representations with equal \(2\)-adic valuations have total signed
weight
$(2^{k-1}-1)\sigma(m)$.
Thus the decomposition provides a structural explanation of an
identity that can also be derived algebraically from the explicit
arithmetic formula for \(a(n)\).

The second main result is the uniform reduction modulo every odd prime
of the auxiliary series \(S(q)\) occurring in the factorization
$A(q)=P(q)S(q)$.
For every odd prime \(p\) and every \(N\ge1\), this reduction yields the
explicit transfer formula
\[
\mathcal W_p(N)
\equiv
\frac{\sigma(N)-\lambda(N)}4
-\frac{\sigma(N/2)}2
+\sigma(N/4)
\pmod p.
\]
Consequently, the corresponding theta-weighted sum involving
\(\operatorname{pod}(n)\) is determined modulo \(p\) by elementary
arithmetic functions. In particular, for \(m\) odd and
\(\kappa\ge1\),
\[
\mathcal W_p(2^\kappa m)
\equiv
\frac{
	(2^{\kappa+1}-3)\sigma(m)-\lambda(m)
}{4}
\pmod p.
\]

When
$m\equiv3\pmod4$
and
$2^{\kappa+1}\equiv3\pmod p$,
the preceding formula gives
\[
\mathcal W_p(2^\kappa m)\equiv0\pmod p.
\]
For \(p=5\) and \(p=13\), this produces the exponent families
$\kappa\equiv2\pmod4$
and
$\kappa\equiv3\pmod{12}$,
respectively. The separate reduction modulo \(3\), combined with the
two families of congruences for \(a(n)\), gives further weighted
congruences involving \(\operatorname{pod}(n)\).

More generally, for every odd prime \(p\), every odd positive integer
\(m\), and every \(\kappa\ge1\), the explicit transfer formula shows
that any arithmetic condition forcing
\[
(2^{\kappa+1}-3)\sigma(m)-\lambda(m)
\equiv0\pmod p
\]
produces a corresponding weighted congruence for
\(\operatorname{pod}(n)\). It would be natural to investigate such
vanishing conditions for other arithmetic progressions and to seek
analogous transfer principles for signed partition functions whose
generating functions admit both divisor-sum evaluations and
factorizations involving classical partition functions.





\bigskip


\end{document}